\documentclass{article}

\usepackage{arxiv}

\usepackage[utf8]{inputenc} 
\usepackage[T1]{fontenc}    
\usepackage{hyperref}       
\usepackage{url}            
\usepackage{booktabs}       
\usepackage{amsfonts}       
\usepackage{nicefrac}       
\usepackage{microtype}      
\usepackage{lipsum}
\usepackage{graphicx}

\usepackage{amssymb}
\usepackage{color}
\usepackage[T1]{fontenc}
\usepackage{amsthm}
\usepackage{amsmath,amsfonts}
\usepackage{enumitem}
\usepackage{appendix}
\usepackage{algorithm}
\usepackage{algcompatible}

\usepackage{hyperref}
\newtheorem{theorem}{Theorem}

\title{Patterns in soil organic carbon dynamics: integrating
microbial activity, chemotaxis and data-driven
approaches}

\author{
 Angela Monti \\
  Istituto per le Applicazioni del Calcolo M. Picone \\
  National Research Council (CNR)\\
  via G. Amendola 122/D, Bari, Italy\\
  \texttt{angela.monti@cnr.it} 
   \And
 Fasma Diele \\
  Istituto per le Applicazioni del Calcolo M. Picone\\
  National Research Council (CNR)\\
  via G. Amendola 122/D, Bari, Italy\\
  \texttt{fasma.diele@cnr.it}
  \And
Deborah Lacitignola \\
  Dipartimento di Ingegneria Elettrica e dell'Informazione\\
  Universita di Cassino e del Lazio Meridionale\\
  via Di Biasio, Cassino, Italy\\
  \texttt{d.lacitignola@unicas.it}
  \And
 Carmela Marangi \\
  Istituto per le Applicazioni del Calcolo M. Picone\\
  National Research Council (CNR)\\
  via G. Amendola 122/D, Bari, Italy\\
  \texttt{carmela.marangi@cnr.it} \\
}

\begin{document}
\maketitle
\begin{abstract}
 Models of soil organic carbon (SOC) frequently overlook the effects of spatial dimensions and microbiological activities. In this paper, we focus on two reaction-diffusion chemotaxis models for SOC dynamics, both supporting chemotaxis-driven instability and exhibiting a variety of spatial patterns as stripes, spots and hexagons when the microbial chemotactic sensitivity is above a critical threshold. We use symplectic techniques to numerically approximate chemotaxis-driven spatial patterns and  explore the effectiveness of the piecewice dynamic mode decomposition (pDMD) to reconstruct them. Our findings show that pDMD is effective at precisely recreating chemotaxis-driven spatial patterns, therefore broadening the range of application of the method to classes of solutions different than Turing patterns. By validating its efficacy across a wider range of models, this research lays the groundwork for applying pDMD to experimental spatiotemporal data, advancing predictions crucial for soil microbial ecology and agricultural sustainability.
\end{abstract}


\section{Introduction}

\noindent 
The presence of soil organic carbon (SOC) is essential to the idea of sustainable soil health management and is necessary for the operation of agro-ecosystems, \cite{Ramesh2019}. High SOC agricultural soils are more capable of retaining water, have greater nutrient availability, and are more resilient to erosion. Furthermore, sequestering SOC can be a major factor in lowering greenhouse gas emissions and lessening the effects of climate change \cite{Chalchissa2024} and modeling  the dynamics of soil organic carbon  may offer important keys to advance comprehension and problem-solving techniques for issues like food security and climate change,  \cite{hammoudi2015mathematical}.

\noindent
The soil organic carbon content is given by the equilibrium of inputs and outputs into the soil system. The primary sources of inputs are the buildup of organic matter (OM) from plants, which can occur as rhizodeposition or litter (plant leftovers). Plants' photosynthetic carbon fixation is the source of this OM. Outputs are related to heterotrophic respiration processes, namely when soil organisms use soil organic carbon as an energy source and release $CO_2$ back into the atmosphere. Throughout the last forty years, soil scientists have tried to identify the involved critical processes and feedbacks in order to model as  faithfully as possible the breakdown of soil organic carbon and the fluxes of $CO_2$. The majority of models consider soil organic matter as divided into a small number of pools (compartments), often from two to five, each with distinct chemical properties or levels of degradability and with decomposition rates mostly determined by the environmental factors, i.e. soil temperature, aeration, and moisture content.  However, although energetic or metabolic processes connected to microbial activities are well recognized to be the source of the majority of organic transformations in soil, most conventional models did not consider the impact of microbial-mediated mechanisms that modify and stabilize soil carbon inputs, \cite{hammoudi2015mathematical}. This issue was addressed by \cite{ Pansu2010} who developed MOMOS (Modelling Organic changes by Micro-Organisms of Soil), a new nonlinear model that explicitly accounted for the microbial biomass activity and was divided into five compartments: microbial biomass (MB), labile and stable fractions of necromass materials and finally labile and stable fractions of humus. 
%
\noindent
Pansu's model -  initially built and validated from a comparative study using isotopic tracers in two sites, \cite{ Pansu2004}  - showed that, in contrast to previous model predictions, microbial respiration was an essential factor in driving microbial dynamics more in accordance with field data.  

\noindent
Although most SOC dynamic models are described by ordinary differential equations, explicitly including space is necessary to progress the quantitative description of microbial carbon usage, since microbial activity is influenced by spatial factors, \cite{Pagel2020}.  Microbial chemotaxis, i.e. the ability of certain bacteria to direct their movement according to the gradient of chemicals contained in their environment, is very common among soil bacteria and is essential for mediating interactions between plants and microbes, \cite{Weigh2023}. Exploring the mechanistic basis of chemotaxis, has been the focus of an enormous amount of theoretical and experimental research. The first mathematical idea to model chemotaxis was developed in 1953 by \cite{Patlak1953}, whereas in the 1970s, \cite{keller1970initiation} were the first to  explore - through reaction–diffusion equations - the dynamics of chemotaxis phenomena, viewing the initiation of slime mold aggregation as instability.  These pioneering works initiated a fruitful period of mathematical analysis for both the Keller–Segel model and chemotaxis, leading to many analytic results on the subject such as aggregations, dynamics of blow-ups, travelling waves and the interesting property of the above system to give rise to spatial pattern formation, \cite{Hillen2008,Horstman2003}.
Furthermore, in addition to the broad and varied research in the field of pattern formation in PDEs, e.g \cite{Murray_book2, Maini_1997, Zhang_2020, Capone_2021, Lacitignola_2022, Grifo_2023, Mehdaoui_2024, Lacitignola_2021}, reaction-diffusion models with chemotaxis turned out to be a precious tool to explain effects that are not observed in reaction diffusion systems with only self-diffusion or cross-diffusion, \cite{Chaplain_1995,Giunta_2021,Bellomo_2022,Bisi_2023}.\\
Among the different models that generalize the scheme proposed by Keller-Segel, we will focus on two reaction-diffusion chemotaxis systems: the spatial MOMOS model  \cite{hammoudi2018mathematical} and the Mimura-Tsujikawa model \cite{MIMURA1996,kuto2012}, both supporting chemotaxis-driven pattern formation.  This feature turns out to be particularly important since it may provide possible explanations for the formation of soil aggregations \cite{Sarker_2022}, for the bacterial and microorganisms spatial organizations (hotspots in soil) or justify the formation of the microscopic patterns observed in \cite{Vogel2014}.  

\noindent
The numerical approximation of spatial patterns, that are solutions of reaction diffusion chemotaxis models, could be very demanding because it might require a large and finely discretized spatial domain and also a long integration time to obtain accurate solutions. In order to reduce the computational cost of a standard numerical method for these class of equations, the application of data-driven techniques is needed. These methods are directly applied to discrete temporal datasets (experimental or numerical data), in principle without the knowledge of a given mathematical model behind. This approach differs from that used by physics-based Model Order Reduction (MOR) methods, that need the model's governing equations. 

\noindent
In this paper we focus on the Dynamic Mode Decomposition (DMD), firstly introduced by \cite{Schmid_2010} and later modified by \cite{Tu_2014}. The DMD technique is a model reduction approach that aims at reconstructing the best linear dynamical system hidden in a given temporal dataset, by fitting a linear model on the data. In \cite{AMS2024}, the authors show that DMD cannot be able to reconstruct accurately datasets corresponding to solutions with different spatio-temporal dynamics, even if large ranks are used in the DMD algorithm. They hence propose a \emph{piecewice} version of the DMD algorithm (pDMD), obtained by splitting the time interval into several subintervals. They show its effectiveness, in terms of accuracy and computational cost, in reconstructing datasets arising from time snapshots of certain reaction-diffusion PDE systems, whose solutions exhibit peculiar dynamics, such as Turing patterns. 

\noindent
We want here to apply the pDMD approach to reaction-diffusion models with chemotaxis to show that this method is able to reconstruct data from a broader class of models than that considered in \cite{AMS2024}.
We have hence chosen to test the validity of the pDMD method using two models that exhibit chemotaxis-driven pattern formation. This first step can be considered as a further validation that preludes the application of pDMD directly to experimental spatio-temporal data of soil microorganisms with the aim not only to reconstruct experimental microbial patterns but also to make possible predictions.

\noindent
The paper is organized as follows: in Sections \ref{RD_MOMOS} and \ref{RD_Mimura} the spatial MOMOS model 
and the Mimura-Tsujikawa model are introduced. For both these reaction-diffusion chemotaxis systems, conditions for chemotaxis-driven instability and spatial pattern initiation are derived as functions of the models parameters and a critical threshold for the chemotaxis sensitivity established. In Section \ref{numer}, symplectic techniques are used to numerically approximate a variety of chemotaxis-driven spatial patterns for both the models and the piecewice dynamic mode decomposition (pDMD) is applied to reconstruct them. In Section \ref{conclusions}, discussion and concluding remarks close the paper.   


\section{The spatial MOMOS model} \label{RD_MOMOS}

\noindent
In  \cite{hammoudi2018mathematical}, a spatial version of the MOMOS model is proposed, explicitly accounting for chemotaxis. The authors prove that spatial patterns cannot emerge in the model without the chemotaxis term even if no analytical conditions - expressed in terms of the system parameters - is provided for the emergence of chemotaxis-driven instability.  

\noindent
In this section, we focus on the dependence of pattern formation on the chemotaxis term, that is here considered as a bifurcation parameter and explicitly derive a critical threshold - expressed as function of the system parameters - for the onset of spatial patterns for the reaction-diffusion chemotaxis MOMOS model.

\noindent
To this aim, we focus on the following model, \cite{hammoudi2018mathematical}, 

\begin{equation}\label{eq:model}
\begin{cases}
\partial_t u - D_u \Delta u &= -\beta \, div( h(u) \nabla v) - k_1 \, u^p - q \, u^2 + k_2 \,   v\,\quad (t, x) \in \Omega_T, \\
\partial_t v - D_v \Delta v &= k_1 \, u^p \, -k_2 \, v  + \, c, \quad (t, x) \in \Omega_T, \\
\nabla u \cdot \mathbf{n} &= \nabla v \cdot \mathbf{n} = 0, \quad (t, x) \in \Sigma_T, \\
u(0) &= u_0 \in \Omega, \\
v(0) &= v_0\in \Omega.
\end{cases}
\end{equation}

\noindent
that represents a variation of the Keller-Segel model with the reaction part modified to fit the MOMOS model with density-dependent microbial turnover, by means of the exponent $p\in [1,2]$, \cite{diele2023soc}. Differently from the general MOMOS model, \cite{Pansu2010}, where the interactions between organic soil and microbial biomass are described through five different compartments, in model (\ref{eq:model}) a simplified version with only two compartments is considered: the microbial biomass MB and the soil organic matter OM, represented by the state variables $u$ and $v$ respectively.  Here the parameter $\beta$ is the chemotaxis sensitivity whereas $D_u$ and $D_v$ are the diffusion parameters of the microbial mass and of the soil organic matter, respectively. 
The continuous function $h(\cdot)$ is involved in the modeling of chemotaxis. As in  (\cite{hammoudi2018mathematical}), we can consider $h(u) = u(M - u)$ if $0 \leq u \leq M$, and zero otherwise in order to prevent overcrowding of microorganisms or $h(u) = u$. 
The parameters involved in the reaction terms have the following meaning: $k_1$ is the microbial mortality rate, $k_2$ is the soil carbon degradation rate; $q$ is the metabolic quotient; $c$ is the soil carbon input. All the parameters in  model (\ref{eq:model})  are assumed to be positive constants. $\Omega$ is a smooth and bounded domain in $R^n$ representing the soil with $\Omega_T=(0,T)\times \Omega$ and $\Sigma_T=(0,T) \times \partial \Omega$. Zero-flux boundary conditions and a proper set of initial conditions are considered.  In the following, we will focus on the case $p=1$ and $h(u)=u$. 

\subsection{Chemotaxis-driven pattern formation in the MOMOS model} \label{cond_instab}

\noindent 
We first observe that model (\ref{eq:model}) admits the two following constant solutions as spatially homogeneous equilibria: $\left(u_i, \displaystyle \frac{u_i(q\,u_i+k_1)}{k_2}\right)$, for $i=1,2$, where $u_1=\sqrt{\frac{c}{q}}$ and  $u_2=-\sqrt{\frac{c}{q}}$.
Because of the biological framework here involved, in the following  will focus on the unique positive and hence feasible spatially homogeneous equilibrium:  

\begin{equation*}
  P_e=(u_1,v_1)=\left( \sqrt{\dfrac{c}{q}}, \dfrac{k_1}{k_2} \sqrt{\frac{c}{q}} +\dfrac{c}{k_2} \right)
 \end{equation*}

 \noindent
 In \cite{hammoudi2018mathematical}, it was shown that no Turing instability and hence no spatial pattern formation can emerge in model (\ref{eq:model}) in the absence of chemotaxis, i.e. for $\beta=0$, stressing that chemotaxis need to be taken into account for the possible emergence of spatial patterns.  In fact, in reaction diffusion systems, directed movements could have the effects of destabilizing a spatially homogeneous solution when chemotaxis is sufficiently large, i.e. when $\beta$ is above a certain threshold, hence causing the onset of spatially inhomogeneous solutions. More precisely, the following result holds: 

 \begin{theorem}
\label{teo1}
Let 

$$
\beta^*= \displaystyle \frac{\sqrt{q}}{k_1\sqrt{c}}\left(D_u\,k_2+D_v\,k_1+2\,D_v\sqrt{cq}+\sqrt{8\,D_u\,D_v\,k_2\,\sqrt{cq}} \right)
$$

\noindent 
If $\beta>\beta^*$, then the spatially homogeneous equilibrium $P_e=(u_1,v_1)$ of model (\ref{eq:model}) undergoes to chemotaxis-driven instability.
\end{theorem}

\begin{proof}
We first observe that the spatially homogeneous solution $P_e=(u_1,v_1)$ is always linearly stable in the absence of diffusion ($D_u=D_v=0$) and chemotaxis ($\beta=0$). In fact, the Jacobian matrix of the reaction terms, evaluated at $P_e$, is given by: 

\begin{equation*}
 \label{J_Momos}
     J^*=\left[ 
     \begin{array}{ccc}
        -\,k_1 -2\sqrt{c\,q} 
        && k_2 \\
        k_1  && -k_2
     \end{array}\right],
 \end{equation*}

\noindent 
so that $detJ^*=2\,k_2\,\sqrt{cq} >0$ and $trJ^*=-2\,\sqrt{cq}-k_1-k_2 <0$. \\
We also recall that, when diffusion is present ($D_u, D_v \neq 0$) but chemotaxis is not considered ($\beta=0$), the spatially homogeneous equilibrium $P_e$ is still linearly stable,  \cite{hammoudi2018mathematical}, so that no diffusion-driven instability can be observed. 
We now consider model (\ref{eq:model}) with both diffusion and chemotaxis. According to the standard linear stability analysis (\cite{Murray_book2}), the stability of $P_e$ is determined by the eigenvalues of the following matrix: 

\begin{equation*}
 \label{Hk_Momos}
     H_k=\left[ 
     \begin{array}{ccc}
        -D_u\,\lambda_k-2\sqrt{c\,q}-k_1 && \beta\,\sqrt{\frac{c}{q}}\,\lambda_k+k_2 \\
        k_1  && -D_v\,\lambda_k-k_2
     \end{array}\right],
 \end{equation*}

\noindent
where $\lambda_k$ is the $k-th$ eigenvalue of the operator $-\Delta$ over $\Omega$ under the no-flux boundary conditions. The characteristic polynomial of (\ref{Hk_Momos}) is given by: 

\begin{equation}
\label{charpoly}
\mu^2+Q_1(\lambda_k)\,\mu + Q_0(\lambda_k)=0,
\end{equation}

\noindent
where

\begin{equation*}
\begin{array}{ll}
Q_1(\lambda_k)=(D_u+D_v)\,\lambda_k+2\,\sqrt{c\,q}+k_1+k_2 >0,\\ 
Q_0(\lambda_k)=A_0\,\lambda_k^2+B_0\,\lambda_k+C_0, 
\end{array}.
\end{equation*}

\noindent
and with 

\begin{equation*}
\begin{array}{ll}
A_0=D_u\,D_v >0;\\
B_0=D_u\,k_2+D_v\,\left(2\,\sqrt{c\,q}+k_1\right)-\beta\,\sqrt{\frac{c}{q}}\,k_1;\\
C_0=2\,k_2\sqrt{c\,q}>0.
\end{array}.
\end{equation*}

\noindent 
Looking at (\ref{charpoly}) we recall that, in order $P_e$ to be unstable, we need $Re(\mu(\lambda_k))>0$ for some $k \in \mathbb N_{+}$. By the Descartes rule of signs, we hence require that $Q_0(\lambda_k)<0$ for some $k \in \mathbb N_{+}$. Since $Q_0(\lambda_k)$ is an upward parabola, this can be obtained by requiring that $B_0<0$ and $\delta^*=B_0^2-4\,A_0\,C_0>0$.\\ 
Observing that, 

\begin{equation}
\label{eq:chemo_condition_momos}
\Bigg \{
\begin{array}{ll}
B_0 < 0 \iff \beta > \tilde{\beta}=\displaystyle \frac{\sqrt{q}}{k_1\,\sqrt{c}}\left(D_u\,k_2+2\,D_v\,\sqrt{c\,q}+D_v\,k_1\right),  \\
\delta^* > 0 \iff \beta > \beta^*=\displaystyle \frac{\sqrt{q}}{k_1\sqrt{c}}\left(D_u\,k_2+D_v\,k_1+2\,D_v\sqrt{cq}+\sqrt{8\,D_u\,D_v\,k_2\,\sqrt{cq}} \right), \\
\end{array}
\end{equation}

\bigskip

\noindent
and holding $\beta^* > \tilde{\beta}$, the thesis easily follows so that $P_e$ changes its stability as the chomotaxis parameter $\beta$ crosses the critical threshold $\beta^*$.
\end{proof}

\noindent
The bifurcation diagram in Figures \ref{fig:momos} and \ref{fig:momos_spots} (Top left panel) shows the region in the $(\beta, q)$ parameter space where conditions (\ref{eq:chemo_condition_momos}) for chemotaxis-driven instability and spatial pattern initiation are verified.   



\section{The Mimura-Tsujikawa model} \label{RD_Mimura}

\noindent
The Mimura-Tsujikawa model was proposed by \cite{MIMURA1996} and analyzed in \cite{kuto2012} to investigate  the pattern dynamics of aggregating regions of biological individuals displaying the chemotaxis proper\-ties. This is given by: 
\begin{equation}
\label{eq:ks}
\begin{cases}
u_t  &= \nabla\cdot (D_u\nabla u-\beta \,u\,\nabla v) + q \,u(1-u),\quad (t, x) \in \Omega_T, \\
v_t &= D_v \Delta v + k_1\, u - k_2 \,v,\quad (t, x) \in \Omega_T, \\
\nabla u\cdot \mathbf{n} &=\nabla v\cdot \mathbf{n} =0\quad x\in\partial\Omega,\\
u(x,0) &= u_0(x), \quad x \in \Omega, \\
v(x,0) &= v_0(x), \quad x \in \Omega
\end{cases}
\end{equation}

\noindent
The state variable $u$ and $v$ here represent a population of biological indivi\-duals and the concentration of a chemical substance, that is produced by the individual itself. The parameter $\beta$ is the chemotaxis sensitivity, $D_u$ and $D_v$ are the diffusion coefficients. The reaction term $f(u) = q\, u(1-u)$ in the first equation is a Fisher type logistic growth with a positive parameter $q$ whereas the reaction term $g(u,v) = k_1 \, v - k_2 \, u$ in the second equation denotes the degradation and the production of the chemical substance, respectively.
We observe that, if there is no growth (i.e. $q = 0$), then the model reduces to the classical Keller–Segel system. 
\noindent
Model (\ref{eq:ks}) admits two spatially homogeneous equilibria, $P_0=(0,0)$ and $P^*=\left(1,\frac{c}{b} \right)$. The Jacobian matrix of the reaction terms is given by
\begin{equation*}
\label{J_ks}
    J=\left[ 
    \begin{array}{ccc}
       q - 2\,q\,u 
       && 0 \\
       k_1 && -k_2
    \end{array}\right]
\end{equation*}
and it is easy to see that, in the absence of any spatial variation,  $P_0=(0,0)$ is linearly unstable  whereas $P^*=(u^*,v^*)$ is always linearly stable. 
\noindent 
However, the spatially homogeneous equilibrium $P^*$ may become linearly unstable when diffusion and chemotaxis are considered hence driving the system towards spatial pattern formation. More precising, with the same arguments as in Subsection \ref{cond_instab},  it can be easily shown that $P^*$ undergoes to chemotaxis-driven instability when the following conditions are verified:
\begin{equation}
\label{eq:chemo_cond_ks}
\begin{cases}
\beta k_1 - D_u \,k_2 - D_v\, q > 0 \\
    \left(\beta k_1 - D_u\, k_2 - D_v\, q \right)^2 - 4 D_u D_v\, q\, k_2 >0,
    \end{cases}
\end{equation}

\noindent 
namely for

    \begin{equation}
        \label{eq:ks_beta_c}
        \beta > \beta_c :=  \frac{D_u \, k_2 + D_v \, q + 2 \sqrt{D_u\, D_v \, q \, k_2 }}{k_1}
    \end{equation}

\noindent
 that is when the chemotaxis parameter $\beta$ is above the critical threshold $\beta_c$. 
The bifurcation diagram in Figure \ref{fig:mimura} (Top left panel) shows the region in the $(\beta, q)$ parameter space where conditions (\ref{eq:chemo_cond_ks}) for chemotaxis-driven instability and spatial pattern initiation are verified.   

\section{Numerical reconstruction by data-driven techniques}
\label{numer}

\noindent
We want now to apply the pDMD method, developed by \cite{AMS2024}, to datasets constructed by considering numerical approximations of the spatial models at different time instances. To this aim,  
we consider the following general reaction-diffusion model with chemotaxis term  
\begin{equation}
\label{eq:general}
    \begin{cases}
        \partial_t u &=  D_u \Delta u  -\beta \, div(u \nabla v) + f(u,v) \\
\partial_t v &= D_v \Delta v + g(u,v) 
    \end{cases}
\end{equation}
with given initial conditions and homogeneous Neumann boundary conditions. 
In order to build the datasets to feed the pDMD method, we first present a numerical approximation of \eqref{eq:general}
in space and time.

\noindent Firstly, we rewrite the chemotaxis in terms of both cross and nonlinear diffusions, exploiting the following  equivalence
$$- \, div(u \nabla v)\, =\, -\Delta(u\, v) + div(\nabla u\, v).$$

\noindent Then, we derive the ODE system arising from the spatial semi-discretization of \eqref{eq:general} by the Method of
Lines (MOL) based on finite differences of second order with step sizes $h_x =
\frac{Lx}{n_x-1}$, $h_y = \frac{L_y}{n_y-1}$ and $n_x$, $n_y$ meshpoints on the rectangular domain $\Omega = [0,L_x] \times [0, L_y]$, suitably modified for assure zero-Neumann conditions on the boundary (see \ref{app_discretization} for more details).  

\noindent This yields to the resulting ODE system
\begin{equation}
    \label{eq:ode}
    \begin{cases}
        \dot{u} &= D_u \, A u + f_\beta(u,v) \\ 
        \dot{v} &= D_v \, A v + g(u,v) \\
    \end{cases}
\end{equation}
where $A$ is the matrix that approximates the Laplacian operator, whereas the function $f_\beta(u,v)$ takes into account both the discretization of the chemotaxis term and the values of the function $f(u,v)$ on the spatial grid. 

\noindent The temporal approximation of the Equation \eqref{eq:ode}
can be done using classical Runge–Kutta and linear multistep methods. Recently, explicit version of some of the above schemes have been implemented in VisualPDE \cite{walker2023visualpde} a flexible, plug-and-play PDE solver that runs in a web browser on a user’s device. In general,
VisualPDE enables the simple, accessible exploration of the features of model  solutions. However, since it relies on general-purpose explicit time-stepping schemes, certain characteristics of special systems (such as those with conserved quantities and positivity preservation \cite{beck2015positivity}, \cite{diele2018positive}) may not be adequately captured, unlike what could be achieved with tailored numerical methods, such as  geometric integrators \cite{hairer2006geometric}. The 
relevance of using geometric numerical integration  for ecological applications has been stated with several examples in \cite{diele2019geometric}. Interestingly, it has been demonstrated that an inadequate time-stepping procedure can mistakenly suggest the emergence of Turing patterns in parameter spaces where the homogeneous steady-state solution remains stable. 

\noindent The effectiveness of symplectic integrators in managing oscillating patterns was observed and examined in \cite{settanni2016devising}, demonstrating the superiority of implicit-symplectic procedures \cite{diele2014imsp} over some widely-used classical approaches.  However, to our knowledge, such procedures have never been applied to reproduce spatially stable patterns resulting from chemotactic movement. For this reason, to generate datasets suitable for use with the pDMD method, we rely on geometric numerical integrators, with the additional aim of exploring their effectiveness within this novel framework.

\noindent For the purposes of our investigation, we limit to use three distinct first-order schemes for approximating the pattern solutions induced by chemotaxis in the ODE system (\ref{eq:ode}).
 The first scheme is the Symplectic Euler method \cite{hairer2006geometric}, which explicitly treats one variable and implicitly handles the other. Treating the variable $v$ implicitly, and explicitly the variable $u$, the Symplectic Euler scheme applied to (\ref{eq:ode}) advances in time according to: 
\begin{equation}
    \label{eq:SE}
    \begin{array}{rcl}
        u_{n+1} &= u_n \,+\,h_t\,  D_u \, A u_n + h_t\, f_\beta(u_n,v_{n+1}) \\ 
         v_{n+1}  &= v_n \,+\, h_t\,  D_v \, A v_{n+1} + h_t \, g(u_n,v_{n+1}) \\
    \end{array}
\end{equation}

\noindent As the explicit treatment of the diffusive term in \eqref{eq:SE} may be source of stiffness, in \cite{settanni2016devising} a modification of the above scheme has been considered. 
The scheme, referred to as IMSP\_IE  proceeds according to 
\begin{equation}
    \label{eq:Eulero_del_piffero}
    \begin{array}{rcl}
        u_{n+1} &= u_n \,+\, h_t\,  D_u \, A u_{n+1} + h_t\, f_\beta(u_n,v_{n+1}) \\ 
         v_{n+1} &= v_n \,+\, h_t\,  D_v \, A v_{n+1} + h_t \, g(u_n,v_{n+1}) \\
    \end{array}
\end{equation}
Finally, we describe the first order IMSP scheme, which has been analyzed and employed in the literature in the context of spatio-temporal chaos \cite{diele2017numerical}. 
The scheme uses an intermediate approximation $\overline{v}$ as follows:
\begin{equation}
    \label{eq:IMSP}
    \begin{array}{rcl}
        \overline{v} &= v_n + h_t\,g(u_n, \overline{v}) \\ 
  \overline{u} &= u_n + h_t\, \,f_\beta( u_n,\overline{v}))\\ 
    u_{n+1}  &= \overline{u} \,+\, h_t  D_u \, A u_{n+1}  \\        v_{n+1}  &= \overline{v} \,+\, h_t  D_v \, A v_{n+1}  
    \end{array}
\end{equation}
Error and stability estimates of a second order IMSP procedure applied in a finite element framework has been provided in \cite{diele2022}. 

\noindent 
Now we combine the theoretical conditions derived in Section \ref{RD_MOMOS} and \ref{RD_Mimura} with the numerical approximation procedures above described,  to numerically obtain different typologies of spatial patterns for the spatial MOMOS and the Mimura-Tsujikawa models. These patterns will be then reconstructed by the mean of the pDMD approach.   

\subsection{Patterns of stripes and spots for the spatial MOMOS model}
\label{sec:momos}

\noindent
We consider the spatial MOMOS model \eqref{eq:model} and assume for the parameters the following numerical values 
\begin{equation}
    \label{eq:momos_parameters}
    D_u = 10^{-3}, \, D_v = 10^{-3}, \, \beta = 0.056,\, k_1 = 0.4, \, k_2 = 0.6, \, q = 0.075, \, c = 10^{-3} 
\end{equation}

\noindent
For this choice conditions (\ref{eq:chemo_condition_momos}) for chemotaxis-driven instability are verified so that spatial pattern initiation can be observed, as shown in the top left panel of Figure \ref{fig:momos}. We consider as initial conditions spatially random perturbation of the homogeneous equilibrium $P_e$
\begin{equation}
    \label{eq:momos_ic}
    u_0(x,y) = u_1 + 10^{-5}\,{\tt rand}(x,y), \, v_0(x,y) = v_1 + 10^{-5}\,{\tt rand}(x,y) 
\end{equation}
with $(x,y) \in \Omega = [0,20] \times [0,20]$. We discretize the spatial domain $\Omega$ with $n_x = n_y = 21$,
such that $n = n_x n_y = 441$ is the total number of grid points. We then integrate in time by using the IMSP scheme \eqref{eq:IMSP} with time step $h_t = 0.1$ until the final time $T = 80000$. 
In the top right panel of Figure \ref{fig:momos}, we show the numerical solution $u$ at the final time of integration. The chemotaxis-driven instability leads to the emergence of stripes patterns. It is worth noting that the same qualitative behaviour has been observed for the variable $v$ (not reported here). 
%
%
\noindent
In the bottom panel of Figure \ref{fig:momos}, we show the behaviour of the spatial mean \eqref{eq:mean} for the variable $u$, 
\begin{equation}
    \label{eq:mean}
    \langle u(t) \rangle := \frac{1}{| \Omega|} \int_{\Omega} u(x,y,t) \, dx \, dy,
\end{equation}
whose time dynamics confirms that the spatial pattern shown in Figure \ref{fig:momos} (top right) is a stationary solution.
\begin{figure}[t]
    \centering
   \includegraphics[scale=0.45] {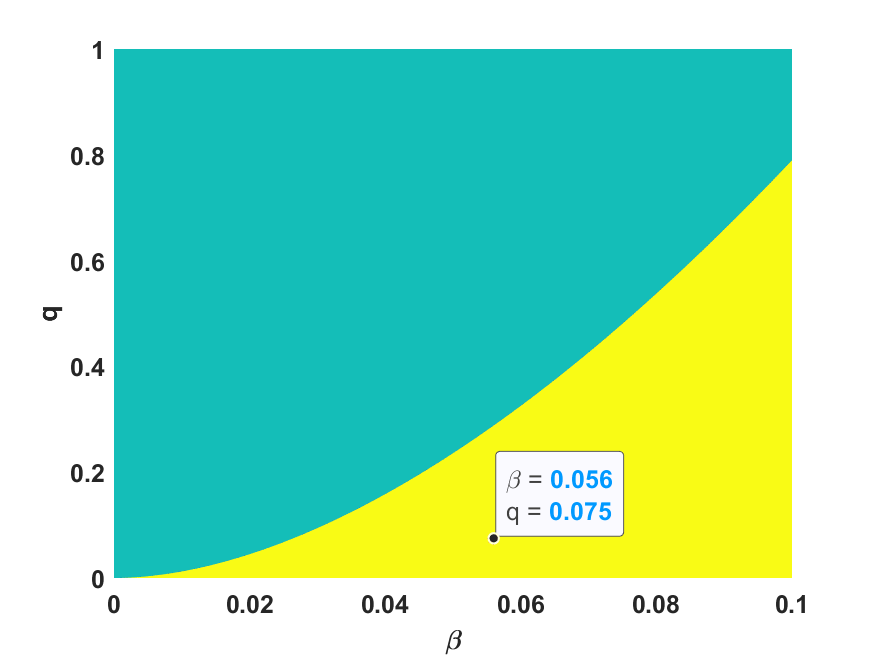} 
    \includegraphics[scale=0.45]{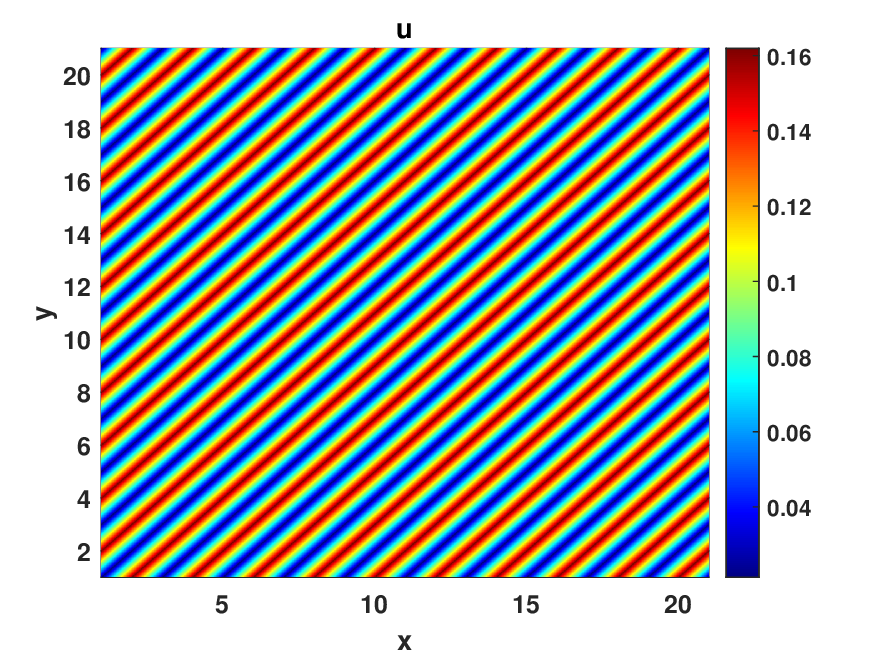}
    \includegraphics[scale=0.45]{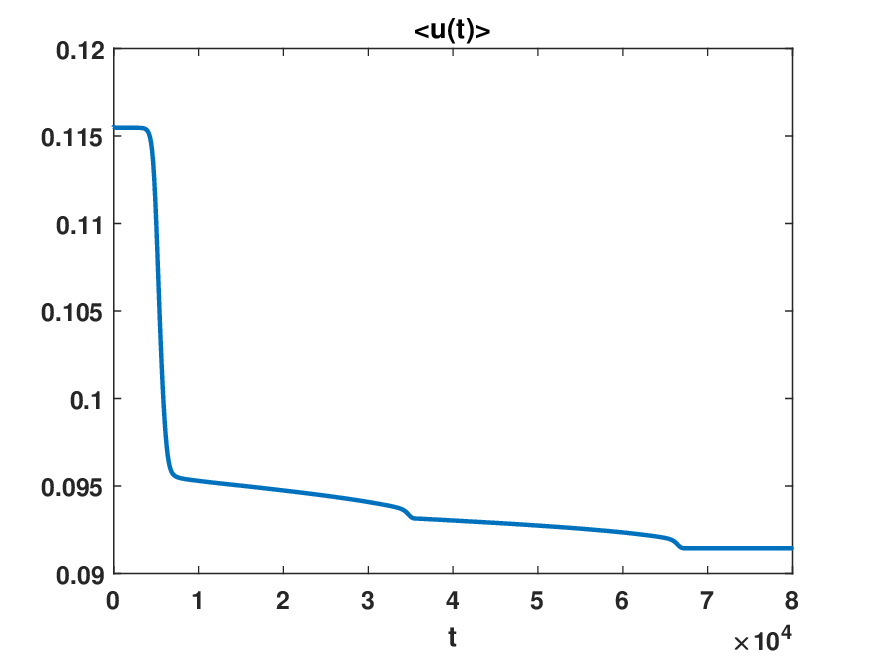}
    \caption{Pattern of \lq \lq stripes\rq \rq for the spatial MOMOS model - Top left plot: bifurcation diagram in the $(\beta,q)$ parameter space. The other parameters are fixed and have the following numerical values: $k_1 = 0.4, k_2 = 0.6, D_u = 10^{-3}, D_v = 10^{-3}, c = 10^{-3}$. The yellow area indicates the region in the parameter space $(\beta, q)$ where conditions \eqref{eq:chemo_condition_momos} are verified whereas parameters in the green area do not satisfy these conditions.
    Top right plot: numerical solution at the final time $T = 80000$ for the state variable $u$. Bottom: time dynamics of the spatial mean of the solutions $u$.} 
    \label{fig:momos}
\end{figure}


\noindent
To obtain a different pattern typology, we consider the same model \eqref{eq:general} with a different choice for the parameters:
\begin{equation}
    \label{eq:momos_param_spots}
    D_u = 0.6, \, D_v = 0.6, \, \beta = 1.2,\, k_1 = 0.4, \, k_2 = 0.6, \, q = 0.075, \, c = 0.8 
\end{equation}
Also in this case conditions \eqref{eq:chemo_condition_momos} are verified so that chemotaxis-driven pattern formation can be observed (see Figure \ref{fig:momos_spots}, top left panel).
We choose as initial conditions 
\begin{equation}
    \label{eq:ic_momos_spots}
    u_0(x,y) = u_1 + 10^{-5}\,{\tt rand}(x,y), \, v_0(x,y) = v_1
\end{equation}
a random perturbation of the first component $u_1$ of the equilibrium $P_e$. The spatial domain $\Omega = [0,25] \times [0,25]$ is discretized with $h_x = h_y = 0.5$, such that $n_x = n_y = 51$ and $n = 2601$ is the number of grid points. The resulting ODE system is integrated by using the IMSP\_IE scheme \eqref{eq:Eulero_del_piffero} with step size $h_t = 0.01$ until $T = 5000$, where spots patterns emerge, as shown in Figure \ref{fig:momos_spots} (top right panel). The temporal behaviour of the spatial mean, reported in the bottom panel of Figure \ref{fig:momos_spots}, confirms that the spots are steady state patterns.

\begin{figure}[htbp]
    \centering
    \includegraphics[scale=0.45]{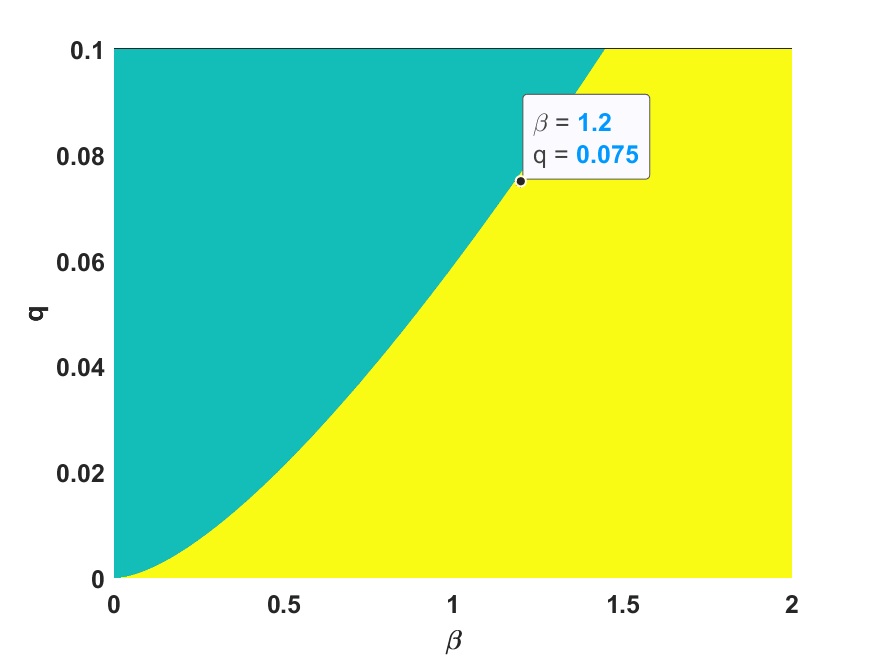}
    \includegraphics[scale=0.45]{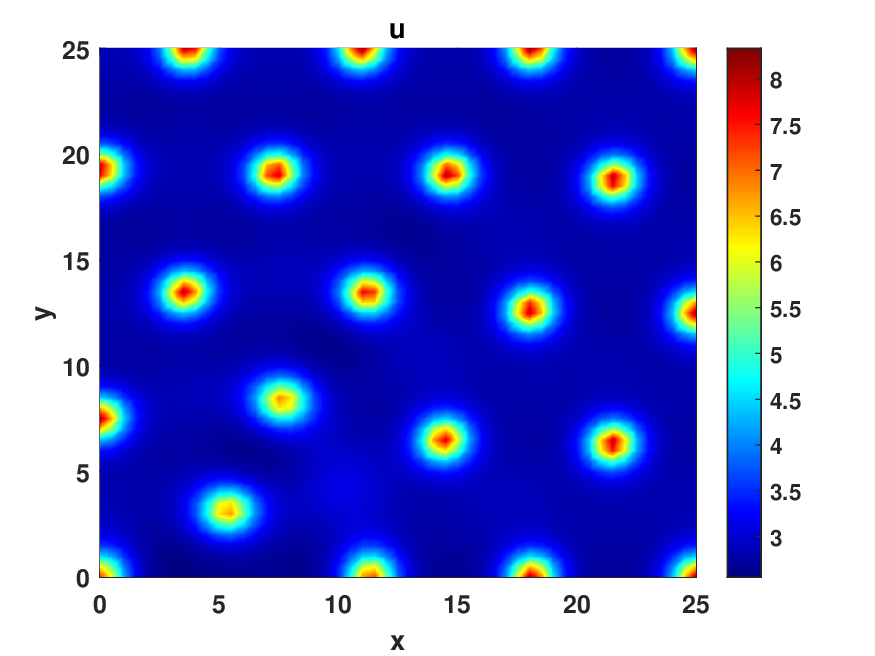}
    \includegraphics[scale=0.45]{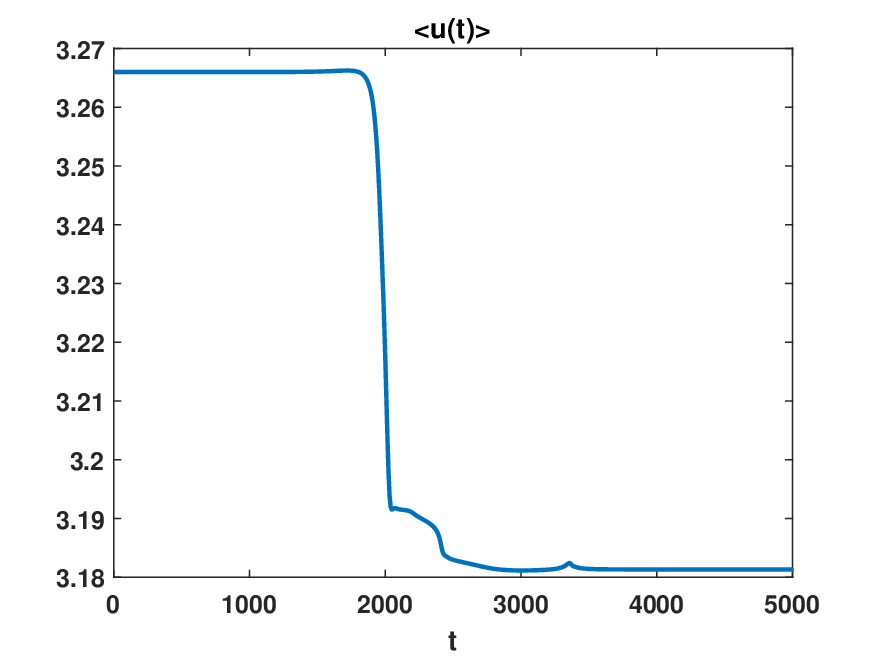}
    \caption{Pattern of \lq \lq spots\rq \rq for the spatial MOMOS model - Top left plot: bifurcation diagram in the $(\beta,q)$ parameter space. The other parameters are fixed and have the following numerical values: $k_1 = 0.4, k_2 = 0.6, D_u = 0.6, D_v = 0.6, c = 0.8$. The yellow area indicates the region in the parameter space $(\beta, q)$ where conditions \eqref{eq:chemo_condition_momos} are verified whereas parameters in the green area do not satisfy these conditions. Top right plot: numerical solution at the final time $T = 5000$ for the state variable $u$. Bottom: time dynamics of the spatial mean of the solutions $u$.}
    \label{fig:momos_spots}
\end{figure}


\subsection{Hexagonal patterns for the Mimura-Tsujikawa model}
\label{sec:mimura}

\noindent
We consider the Mimura-Tsujikawa model \eqref{eq:ks} and assume for the parameters the following set of numerical values, \cite{kuto2012}
\begin{equation}
    \label{eq:mimura_parameters}
    D_u = 0.0625, \, D_v = 1, \, \beta = 17, \, k_1 = 1, \, k_2 = 32, \, q = 7.
\end{equation}

\noindent
For this choice conditions \eqref{eq:chemo_cond_ks} are verified
and the emergence of spatial patterns induced by chemotaxis can be observed, as stressed by the bifurcation diagram depicted in the top left panel of Figure \ref{fig:mimura}. We choose as initial conditions %
\begin{equation}
    u_0(x,y) = u^* + 0.05 \, {\tt rand}(x,y), \, v_0(x,y) = v^*
\end{equation}
with $(x,y) \in \Omega = [0,3] \times [0,3]$. The spatial domain $\Omega$ is discretized by using $n_x = n_y = 50$ meshpoints, such that $n = 2500$. We use the Symplectic Euler scheme \eqref{eq:SE} to integrate the resulting ODE system, with time step $h_t = 10^{-3}$ and final time $T = 500$.
\begin{figure}[t]
    \centering
    \includegraphics[scale=0.45]{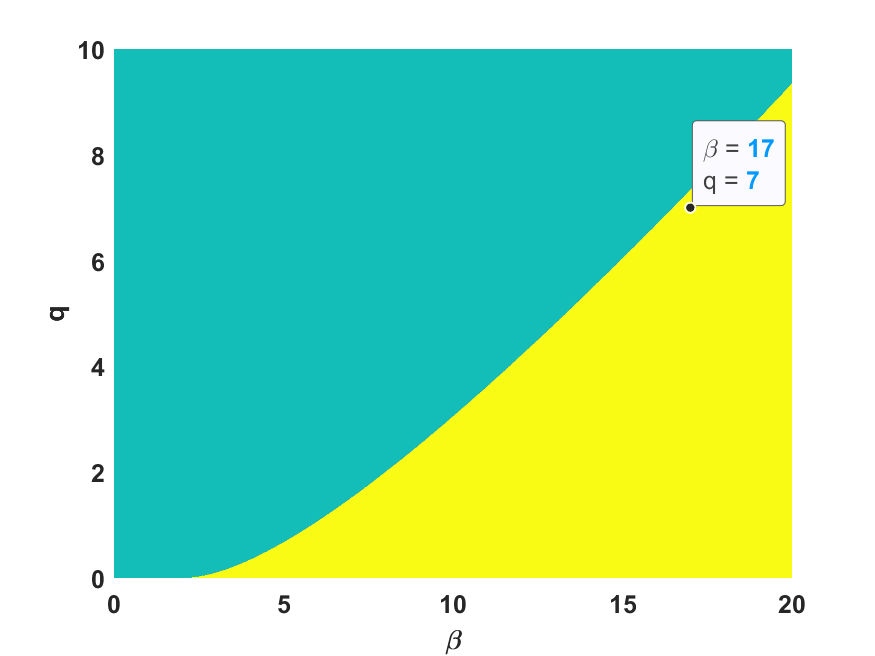}
    \includegraphics[scale=0.45]{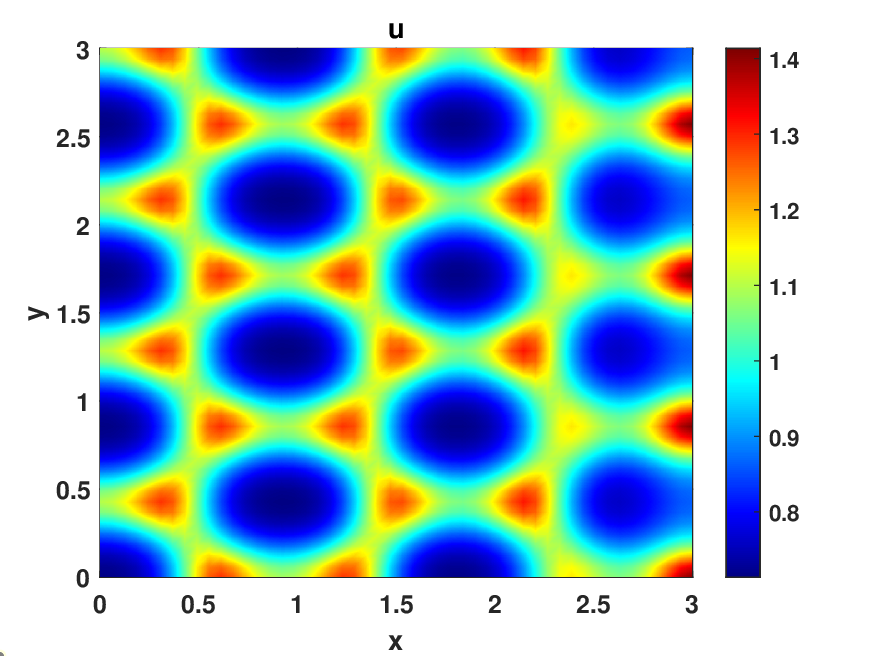}
    \includegraphics[scale=0.45]{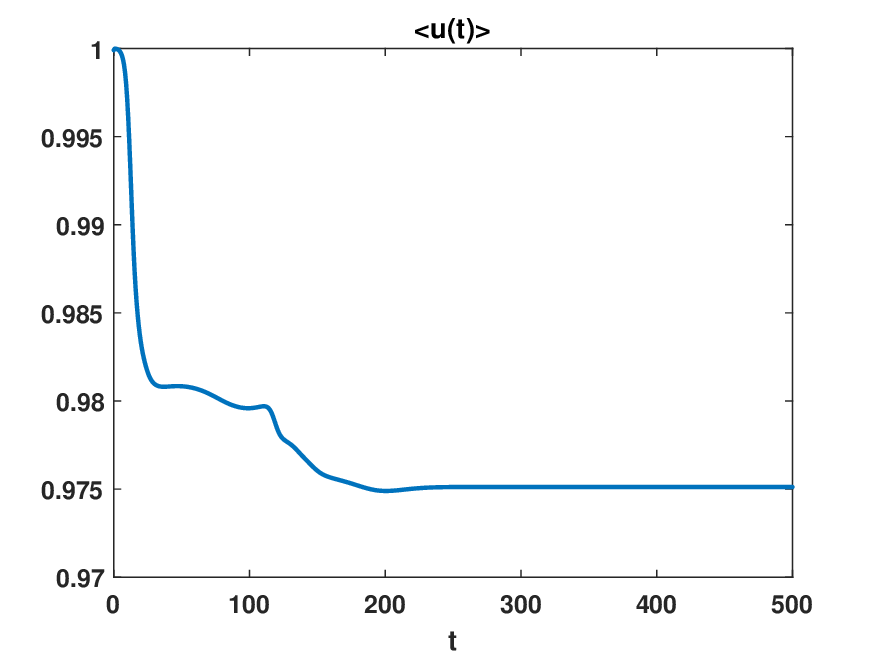}
    \caption{Hexagonal spatial patterns for Mimura-Tsujikawa model - Top left plot: bifurcation diagram in the $(\beta,q)$ parameter space. The other parameters are fixed and have the following numerical values: $k_1 = 1, k_2 = 32, D_u = 0.0625, D_v = 1$. The yellow area indicates the region in the parameter space $(\beta, q)$ where conditions \eqref{eq:chemo_cond_ks} are verified whereas parameters in the green area do not satisfy these conditions. Top right plot: numerical solution at the final time $T = 500$ for the state variable $u$. Bottom: time dynamics of the spatial mean of the solutions $u$.}
    \label{fig:mimura}
\end{figure}
The hexagonal spatial patterns obtained at the final time $T = 500$ are depicted in Figure \ref{fig:mimura} for the unknown $u$ (top right panel). A similar behaviour has been found also for the variable $v$, that is not reported.
%
%
As for the MOMOS model, the time dynamics of the spatial mean, reported in the bottom panel of Figure \ref{fig:mimura} for $u$, indicates that a stationary solution has been obtained.

\noindent 
We want now to apply the pDMD procedure to reconstruct the different class of chemotaxis-driven spatial patterns obtained in the two previous subsections. 

\subsection{Applying the pDMD procedure}
\noindent
Given a snapshot matrix $X$, whose columns are experimental or numerical data at different time instances, the DMD method, firstly introduced by \cite{Schmid_2010}, tries to identify the underlying process by fitting the dynamics hidden in the snapshots, as accurate as possible in a linear sense. Later, \cite{Tu_2014} proposed the \emph{exact} DMD based on a low rank approximation of the original algorithm. Afterwards, many other modifications have been proposed to improve the DMD method (see e.g. \cite{HH17, AK2018,EDMD2015,LV17}). 

\noindent
Here, we consider the application of the pDMD algorithm, introduced by \cite{AMS2024} to deal with datasets that exhibit a peculiar dynamics that the DMD method is not able to reconstruct accurately. First of all, we build the snapshot matrix (the dataset) that is the input of the pDMD algorithm.\\
Let $u_k$ and $v_k \in \mathbb{R}^n$ be the numerical solutions of \eqref{eq:general} at time instances $t_k$, with $k = 1, \dots, m$ and $t_k = t_1 + (k-1) h_t$. We construct the snapshot matrix 
\begin{equation}
    \label{eq:snapshot}
    X = \left[ \begin{array}{cccc}
         | & | & \dots & | \\
u_1 & u_2 & \dots & u_{m} \\
| & | & \dots & | \\
v_1 & v_2 & \dots & v_{m} \\
| & | & \dots & | \\
\end{array}\right] \in \mathbb{R}^{2n \times m},
\end{equation}
whose columns are obtained by concatenating the unknowns $u$ and $v$ at each time instance. This approach, also known as \emph{coupled} \cite{VBGRC22}, reconstructs all together the variables involved in the model, so that each unknown has a knowledge of the other one and vice versa. It is worth noting that one could also consider an \emph{uncoupled} approach, that reconstructs separately each variable $u$ and $v$, without taking into account any possible relation between the unknowns.

\noindent
The main assumption underlying the original DMD algorithm consists in a \lq \lq global\rq \rq linear fitting overall the time interval. Instead of considering the whole time interval, the pDMD method splits $[0,T]$ in $N \geq 1$ parts and makes use of the DMD algorithm in each obtained subinterval. The method is summarized in Algorithm \ref{alg:pDMD} and explained below.

\begin{algorithm}[bth]
\caption{Piecewise DMD (pDMD)}
\label{alg:pDMD}
\begin{algorithmic}[1]
\STATE {\bf INPUT} Dataset $X \in R^{2n \times m}$ in $[0,T]$, initial number $N$ of partitions, thresholds $tol >0$ and $\overline{tol} > 0$
\vspace{0.1cm}
\STATE {\bf OUTPUT} $\widetilde{ X}^{N}$ piecewise reconstruction in $[0,T]$ 
\STATE Split the datasets $X = [X_1, \ldots, X_i, \ldots, X_N]$
\FOR{$i = 1, \ldots, N$}
\STATE set the target rank $r_i = \mbox{rank}(X_i)$  
\STATE compute the (randomized) DMD solution $\widetilde{X}_i$ of rank $r_i$ using Algorithm 3.1 in \cite{EMKB2019}
\STATE compute the error $err(i)$ defined in \eqref{e_max} 
\IF{ $err(i) > tol$}
\STATE    $N = N+1$ 
\STATE   go to step 3
\ENDIF
\ENDFOR
\STATE compute the error $\mathcal{E}(N)$ defined in \eqref{err_frob}
\IF{ $\mathcal{E}(N) > \overline{tol}$}
\STATE    $N = N+1$ 
\STATE   go to step 3
\ENDIF
\STATE $\widetilde{X}^{N} = [\widetilde{X}_1, \ldots, \widetilde{X}_N]$
\end{algorithmic}
\end{algorithm}

\noindent
Given a dataset $X \in \mathbb{R}^{2n \times m}$ and an initial number $N$ of partitions, the pDMD Algorithm \ref{alg:pDMD} starts by splitting the snapshot matrix $X$ in submatrices $X_i \in \mathbb{R}^{2n \times \nu}$, for $i = 1, \ldots, N$, where $\nu = \left[\frac{m}{N} \right]$ is the number of snapshots in $S_i$. In other words, the $i$-th submatrix $S_i$ corresponds to the time subinterval $\left[t_{(i-1) \nu + 1}, t_{i \nu} \right]$. Then, for $i = 1, \ldots, N$, we consider $r_i = {\tt rank}(X_i)$ and compute the (randomized) DMD \cite{EMKB2019} of $X_i$, to obtain a reconstruction $\widetilde{X}_i$. Then we compute the relative error
\begin{equation}
\label{e_max}
err(i) = \max_{(i-1)\nu +1 \leq k \leq i \nu} \frac{\|x_k - \tilde{x}_k \|_{\infty}}{\| x_k \|_{\infty}}
\end{equation}
that is the worst approximation in the time subinterval $\left[t_{(i-1) \nu + 1}, t_{i \nu} \right]$. If this error is greater than a desired chosen tolerance $tol$, then the pDMD algorithm increases the value of $N$ and restart with a new finer partition. Moreover, to ensure that the Algorithm returns a good reconstruction of the dataset, we will also consider
the relative error in Frobenius norm between the dataset $X$ and its piecewise DMD reconstruction $\widetilde{X}^N$ defined by
\begin{equation}
    \label{err_frob}
    \mathcal{E}(N) = \frac{\|X-  \widetilde{X}^N\|_F}{\|X\|_F},
\end{equation}
that depends on the number $N$ used to split the whole dataset $X$. If this error is greater than a chosen threshold $\overline{tol}$, then the pDMD increases the number $N$ of subintervals and restart with this new partition. The Algorithm \ref{alg:pDMD} returns the reconstruction $\widetilde{X}^N$. 

\noindent
In our numerical simulations, we will show the solution at the final time reconstructed by the pDMD method for a number $N$ of partitions of the time interval. Moreover, we will also consider the relative error over time, defined as
\begin{equation}
\label{error_time}
\epsilon(N,t_k) = \frac{\| x_k - \tilde{x}_k \|_F}{\| x_k \|_F}, \quad k = 1, \dots, m
\end{equation}
where $\{ \tilde{x}_k \}_{k=1}^m$ are the snapshots reconstructed by the pDMD with $N$ partitions.

\noindent
First we consider the numerical solutions obtained in Section \ref{sec:momos} by the MOMOS model with parameters as in \eqref{eq:momos_parameters}. We construct the dataset by saving the snapshots every eight time steps, such that $X \in \mathbb{R}^{2n \times m}$, with $n = 441$ and $m = 100000$. We apply the pDMD method by choosing $N = 1$ and thresholds $tol = 10^{-1}$ and $\overline{tol} = 10^{-3}$ as input of Algorithm \ref{alg:pDMD}. By incrementing the partition size $N$ by one, the pDMD Algorithm \ref{alg:pDMD} returns a reconstruction $\widetilde{X}^N$ with $N = 8$, that is the first $N$ value that satisfies the condition $\mathcal{E}(N) \leq \overline{tol}$ (see step 14 of Algorithm \ref{alg:pDMD}). 

\begin{figure}[htbp]
    \centering
     \includegraphics[scale=0.45]{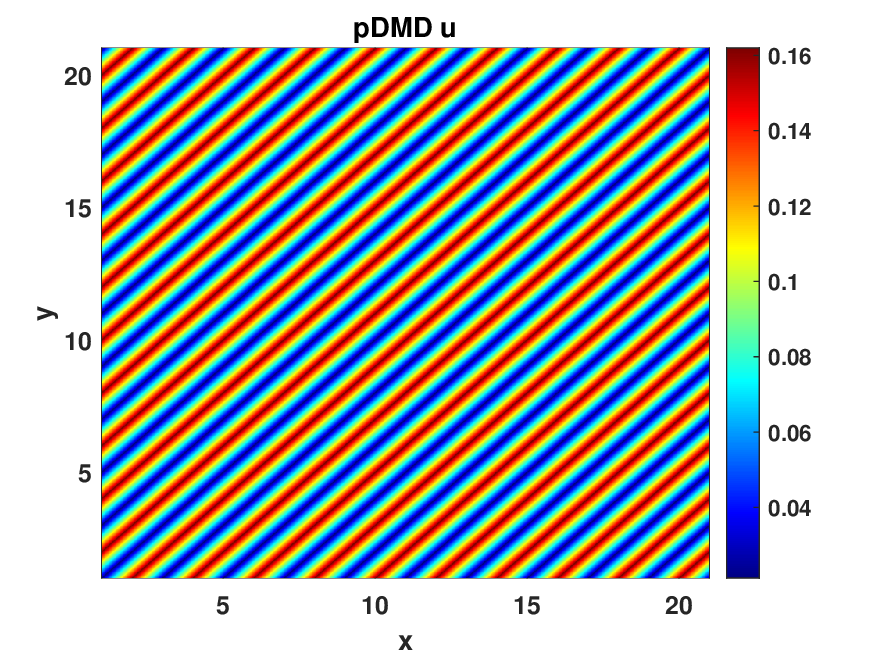}
    \includegraphics[scale=0.45]{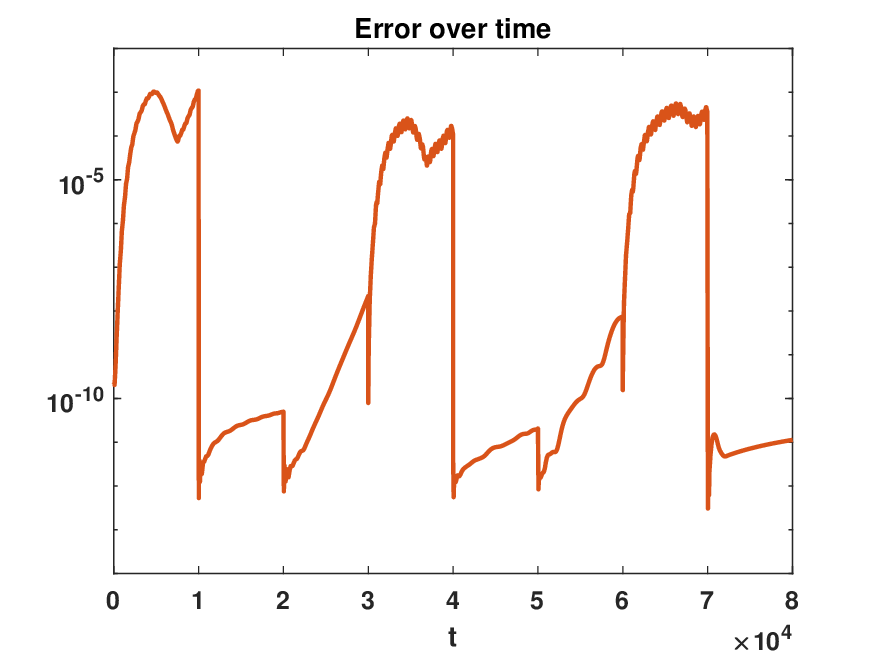}
    \caption{Reconstructed stripes patterns (spatial MOMOS model). Left plot: pDMD reconstruction of the state variable $u$ at the final time $T = 80000$. The number of partition used in the pDMD algorithm is $N = 8$. Right plot: relative error behaviour \eqref{error_time} over time.}
    \label{fig:momos_pdmd}
\end{figure}
\noindent
In the left panel of Figure \ref{fig:momos_pdmd}, we show the spatial pattern at the final time $T = 80000$ reconstructed by the pDMD method for the variable $u$. It is evident that pDMD carefully matches the microbial mass $u$ at the final time.
%

\noindent
Moreover, in the right panel of Figure \ref{fig:momos_pdmd}, we report the relative error \eqref{error_time} over time of the pDMD reconstruction $\widetilde{X}^8$. The error exhibits an erratic behaviour, with peaks at the end of each subinterval $\left[t_{(i-1) \nu + 1}, t_{i \nu} \right]$. Nevertheless, it is less than $10^{-3}$ at each time step. In particular, the error for the variable $u$ at the final time is $7.0884 \times 10^{-12}$. Finally, we also compare the results in terms of computational cost. The IMSP scheme needs about $309.8045$ seconds to solve system \eqref{eq:general}, whereas the pDMD (with $N = 8$ intervals) takes $0.9055$ seconds to reconstruct the solution, with a speed up factor of $342.14$.  \\

\noindent
To show that the pDMD procedure can reconstruct a different typology of patterns, we also consider the spatial MOMOS model with parameters as in \eqref{eq:momos_param_spots}. We recall that, for this choice,  patterns of spots emerge. The dataset is built by saving the snapshots, computed in Section \ref{sec:momos}, every five time steps, such that $X \in \mathbb{R}^{5202 \times 100000}$. We choose $N = 1$ and tolerances $tol = 10^{-1}$ and $\overline{tol} = 10^{-4}$ as input for the pDMD algorithm, that stops when $N = 32$, that corresponds to the first $N$ value for which condition $\mathcal{E}(N) \leq \overline{tol}$ is verified. Thus, it returns the reconstruction $\widetilde{X}^{32}$. In Figure \ref{fig:momos_spots_pdmd}, we show the pDMD reconstruction of the microbial mass $u$ at the final time. The pDMD procedure approximates quite well the shape of the final pattern and also its amplitude (see the colorbar and compare it with Figure \ref{fig:momos_spots}, left plot). As for the previous example, in the right panel of Figure \ref{fig:momos_spots_pdmd}, we show the behavior of the error $\epsilon(N,t_k)$ \eqref{error_time} along the integration
time interval of the PDE model, for $N = 32$. The error is very low in the first and last part of the time interval, attaining an order of about $10^{-12}$, except for peaks of maximum ($\approx 10^{-3}$) that seem to correspond to the major variability in the spatial mean (for $t \approx 2000$, see Figure \ref{fig:momos_spots} right). From a computational point of view, the pDMD with $N = 32$ is able to reconstruct the solution by taking $1.6908$ seconds, compared to the $1.6879 \times 10^3$ seconds needed by the IMSP\_IE scheme, with a speed up factor of $998.34$. 

\begin{figure}[tbhp]
    \centering
     \includegraphics[scale=0.45]{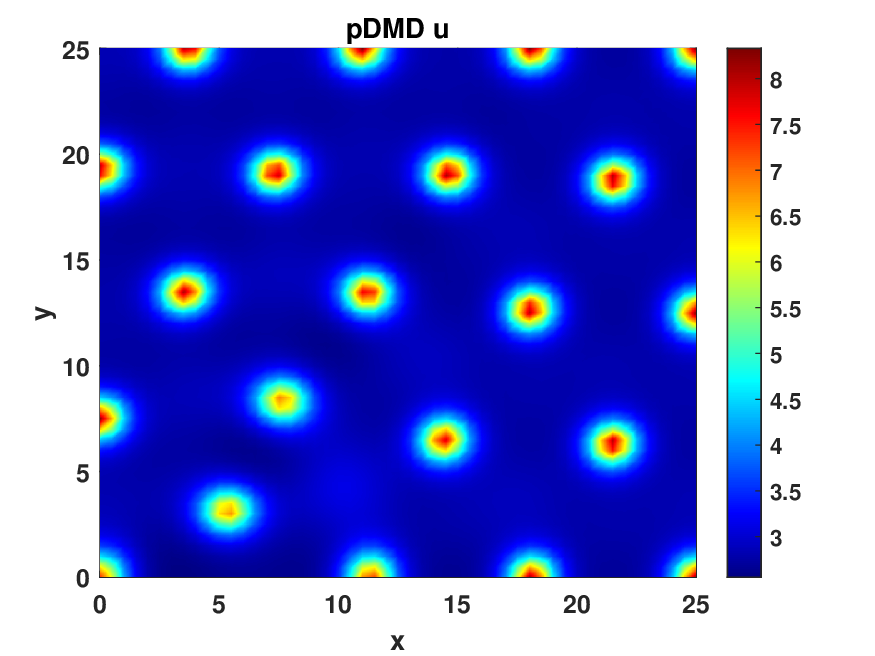}
        \includegraphics[scale=0.45]{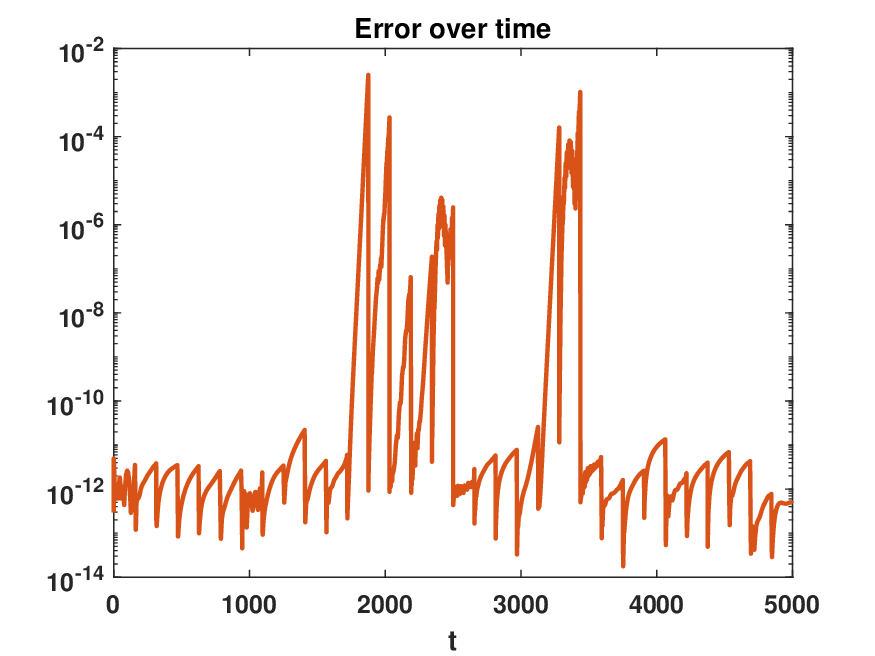}
    \caption{Reconstructed spots patterns (spatial MOMOS model). Left plot: pDMD reconstruction of the state variable $u$ at the final time $T = 5000$, obtained with $N = 32$ partitions. Right plot: relative error behaviour \eqref{error_time} over time.}
    \label{fig:momos_spots_pdmd}
\end{figure}


\noindent
Finally we want to reconstruct the hexagonal patterns of the Mimura-Tsujika\-wa model, obtained in Section \ref{sec:mimura} when the numerical values of the parameters are chosen as in \eqref{eq:mimura_parameters}. To this aim, we construct the dataset $X$ by saving the snapshots every five time steps, such that the snapshot matrix $X$ has dimension $2n \times m$, with $n = 2500$ and $m = 100000$. To start the pDMD algorithm, we choose $N=1$ and thresholds $tol = 10^{-1}$ and $\overline{tol} = 10^{-5}$. The pDMD algorithm stops when $N = 15$, that is the first $N$ value that verifies the condition $\mathcal{E}(N) \leq \overline{tol}$.

\noindent
The pDMD reconstruction of the state variable $u$ at the final time $T = 5000$ is reported in Figure \ref{fig:mimura_pdmd}, left plot. The hexagons are approximated quite well by the pDMD both in shape and in amplitude (see the colorbar and compare it with Figure \ref{fig:mimura}, left plot). As for the previous examples, in the right
panel of Figure \ref{fig:mimura_pdmd}, for $N = 15$, we show the behavior of the relative error \eqref{error_time} along the
integration time interval. The error is greater in the first part, where the spatial mean is rapidly decreasing from the initial time. Then, for $t \geq 135$, the error decreases and attains almost an order of $10^{-12}$. Finally, we compare the results with the symplectic first order integrators, in terms of the CPU time (in seconds). The pDMD needs $1.4396$ seconds, whereas the Symplectic Euler scheme takes $473.9948$ seconds, with a speed up factor of $329.25$.
\begin{figure}[htbp]
    \centering
    \includegraphics[scale=0.45]{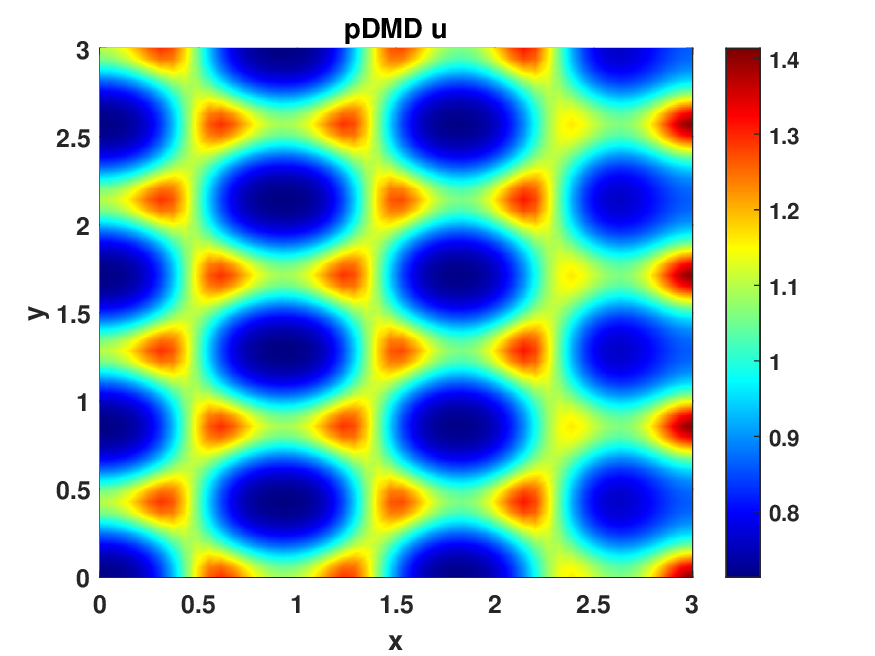}
        \includegraphics[scale=0.45]{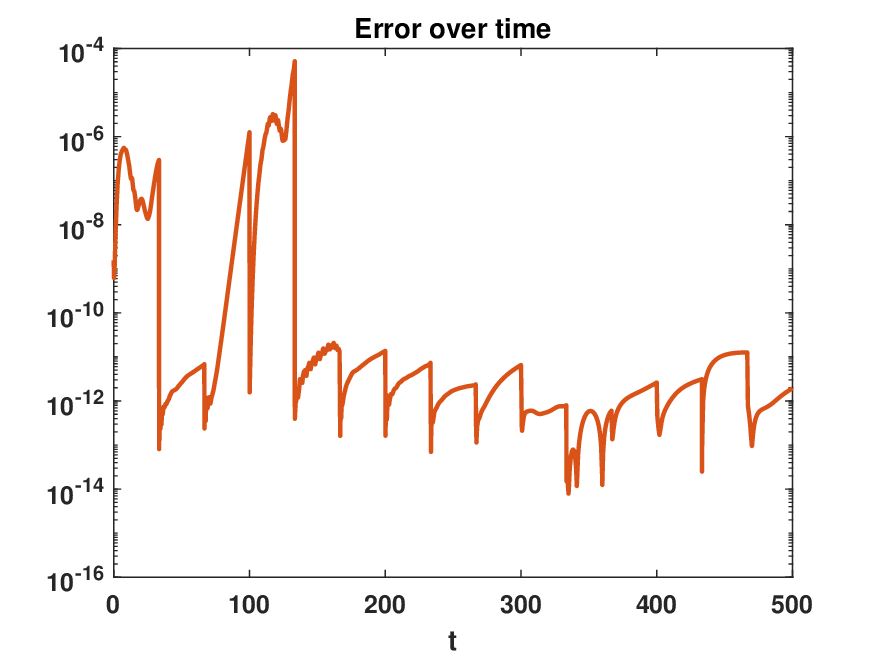}
    \caption{Reconstructed hexagonal spatial patterns (Mimura-Tsujikawa model). Left plot: pDMD reconstruction of the state variable $u$ at the final time $T = 500$, obtained with $N = 15$ partitions. Right plot: relative error behaviour \eqref{error_time} over time.}
    \label{fig:mimura_pdmd}
\end{figure}
\section{Discussion and concluding remarks}
\label{conclusions}

\noindent In this paper, we have addressed the numerical challenges related to spatial organization phenomena in two reaction-diffusion chemotaxis systems: the spatial MOMOS model \cite{hammoudi2018mathematical}  and the Mimura-Tsujikawa model \cite{MIMURA1996, kuto2012}, both supporting chemotaxis-driven pattern formation. We have shown how symplectic procedures could accurately reproduce spatial patterns like spots, stripes and hexagons generated by bacteria. Even if our numerical tests focus solely on the results obtained through symplectic methods, some preliminary comparisons with classical approaches, as IMEX methods, revealed a more favorable behavior of the symplectic methods. Such a point certainly deserves to be further explored from a theoretical point of view. To this aim we want to follow the approach considered in previous works \cite{settanni2016devising}, and focussed on Turing patterns \cite{sgura2012numerical, diele2019geometric}, to derive conditions necessary for the emergence of chemotaxis-driven patterns within a discrete framework. This will also allow us to  establish whether and when the symplectic structure of the discrete flow can offer any advantages. 

\noindent
Even though the symplectic schemes could accurately approximate spatial patterns emerging from chemotaxis-driven instability, their numerical approximation could be very demanding. For this reason, by following the approach in \cite{AMS2024}, we have considered the application of the pDMD to the reaction-diffusion models with chemotaxis showing that very accurate reconstructions can be obtained by this data-driven method. In addition, the application of the pDMD algorithm turned out to be more efficient in terms of the computational cost. In this study we have applied the pDMD method to datasets consisting of numerical approximations of the PDE models at difference time instances. However, one of the main strengths of data-driven techniques is that they do not require necessarily the knowledge of the model's governing equations so that they could be directly applied to field data and could provide future predictions for moderate time horizons (see e.g. \cite{AMS2024}). For this reason, our further step will be to apply data-driven methods to experimental data, such as spatio-temporal data of soil microorganism, to reproduce observed spatial patterns and possibly gain predictions on their spatio-temporal dynamics. 
This may be especially significant since it is difficult to generalize observed patterns and obtain a predictive understanding of the dynamics of soil microbial communities because of the interaction of geochemical and biological processes, as well as the effects of particular environmental conditions \cite{Patoine_2022}. Following that, data-driven spatio-temporal models of soil biological communities can be a precious tool in  identifying areas of the world that may undergo changes and in creating management and conservation plans that are suitable for those areas.

\noindent
Developing and implementing such models not only leads to a theoretical advancement in the understanding of the many actors, feedbacks and interactions involved in the SOC dynamics   but can also open the way to test innovative and challenging solutions in terms of sustainability.  For example, the importance of chemotaxis-driven patterns and overall spatial organization of soil microorganisms is increasingly recognized across various fields, including circular economy. One of the most promising applications of soil-dwelling microorganisms lies in fact in the ability of some bacteria (Pseudomonas putida) to contribute to the breakdown and assimilation of polyethylene terephthalate (PET), a ubiquitous plastic widely used in food packaging. Research endeavors aim not only to explore the chemical capabilities of these microorganisms but also to improve the efficiency of plastic degradation by leveraging diverse microbial populations and optimizing spatial arrangements \cite{Bao}. Understanding the dynamic distribution of microbial populations in space is crucial in this context, as it can either impede or facilitate the process of plastic deconstruction. Moreover, SOC is crucial  for adapting to expected changes in climate and land use, being at the same time a carbon sink and a way to reduce future $CO_2$ emissions. In fact, if improved farming practices were implemented, soils could absorb more than 10\% of the world's anthropogenic emissions. Governments and stakeholders are then becoming more interested in encouraging farmers to implement cultivation techniques that improve carbon sequestration as part of larger climate change mitigation strategies. Carbon farming is one of them \cite{Singh_2024}. It is a new, sustainable technique that uses photosynthesis to trap greenhouse gases, utilizing environmentally friendly methods like low tillage, agroforestry, cover crops and rotational grazing hence improving soil conservation. Additionally, by incorporating the circular carbon economy's tenets, carbon farming makes it possible to monetize carbon credits, encouraging startups to transform agricultural resources into bio-based goods and hence promoting a circular and sustainable development.


\section*{Code availability}
\noindent
The MATLAB source code for the implementations used to compute the presented results can be downloaded from \url{https://github.com/CnrIacBaGit/PatternChemotaxis}

\section*{Acknowledgements}

\noindent A.M., F.D. and  C.M. research activity is funded by the National Recovery and Resilience Plan (NRRP), Mission 4 Component 2 Investment 1.4 - Call for tender No.
3138 of 16 December 2021, rectified by Decree n.3175 of 18 December 2021 of Italian Ministry of University and Research funded by the European Union – NextGenerationEU; Award Number: Project code CN 00000033, Concession Decree No. 1034 of 17 June 2022 adopted by the Italian Ministry of University and Research, CUP B83C22002930006, Project title \lq \lq National Biodiversity Future Centre.
D.L. research has been developed within the Project ”P2022PSMT7” CUP H53D23008940001 funded by EU in NextGenerationEU plan through the Italian "Bando Prin 2022 - D.D. 1409 del 14-09-2022" by MUR.

\noindent F.D. and A.M. are members of the INdAM research group GNCS; D.L. is member of the INdAM research group GNFM. F.D. and C.M. would like to thank Mr. Cosimo Grippa for his valuable technical support. 

\appendix

\section{Spatial discretization}
\label{app_discretization}
\noindent
In this section we focus on the spatial semi-discretization of system \eqref{eq:general}. We consider a rectangular domain  $\Omega = [0,L_x] \times [0, L_y]$, with $n_x$, $n_y$ meshpoints, such that the stepsizes are $h_x =
\frac{Lx}{n_x-1}$, $h_y = \frac{L_y}{n_y-1}$. We compute the first and second order derivatives on the spatial grid by using central finite differences
\begin{equation}
\begin{aligned}
    \frac{\partial u(x,y)}{\partial x} &= \frac{u(x+h_x,y)-u(x-h_x,y)}{2 \, h_x}, \\
    \frac{\partial u(x,y)}{\partial y} &= \frac{u(x,y+h_y)-u(x,y-h_y)}{2 \, h_y},
    \end{aligned}
\end{equation}

\begin{equation}
    \begin{aligned}
    \frac{\partial^2 u(x,y)}{\partial x^2} &= \frac{u(x+h_x,y)-2 u(x,y) + u(x-h_x,y)}{h_x^2}, \\
    \frac{\partial^2 u(x,y)}{\partial y^2} &= \frac{u(x,y+h_y)-2 u(x,y) + u(x,y-h_y)}{h_y^2}.
    \end{aligned}
\end{equation}
We approximate the homogeneous Neumann boundary conditions as follows
\begin{equation}
\begin{aligned}
    \frac{\partial u (0,y)}{\partial x} &= \frac{u(0+h_x,y)-u(0-h_x,y)}{2\,h_x} = 0, \\
    \frac{\partial u (L_x,y)}{\partial x} &= \frac{u(L_x+h_x,y)-u(L_x-h_x,y)}{2\,h_x} = 0 
    \end{aligned}
    \end{equation}
    for the variable $u$ along the $x$ direction and, similarly along $y$.
These yield
    \begin{equation}
   \begin{aligned}
    u(-h_x,y) &= u(h_x,y), \\
    u(L_x+h_x,y) &= u(L_x-h_x,y),
    \end{aligned} 
\end{equation}
by using fictitious nodes $u(-h_x,y)$ and $u(L_x+h_x,y)$. Similarly, we obtain conditions for the variable $v$. \\
Thus the discretization of the diffusion term $\Delta u$
is given by
\begin{equation}
\begin{aligned}
    & \Delta u(x,y) = \frac{\partial^2 u(x,y)}{\partial x^2} + \frac{\partial^2 u(x,y)}{\partial y^2} \\ &= \frac{u(x+h_x,y)-2u(x,y)+u(x-h_x,y)}{h_x^2} + \frac{u(x,y+h_y)-2u(x,y)+u(x,y-h_y)}{h_y^2}.
    \end{aligned}
\end{equation}
Then we consider the following equivalence
$$- \, div(u \nabla v)\, =\, -\Delta(u\, v) + div(v\,\nabla u),$$
that allows us to rewrite the chemotaxis term as sum of a \emph{cross diffusion term} $\Delta(u\,v)$ and a \emph{nonlinear diffusion term} $div(v \, \nabla u)$. As regards the discretization of the last term, along the $x$ direction we have
\begin{equation}
    \begin{aligned}
        \frac{\partial}{\partial x} (v u_x) &= \frac{\left(v(x,y)+v(x+h_x,y)\right)}{2 h_x} \frac{\left(u(x+h_x,y)-u(x,y)\right)}{h_x} \\
        &- \frac{\left(v(x,y)+v(x-h_x,y)\right)}{2 h_x} \frac{\left(u(x,y)-u(x-h_x,y)\right)}{h_x} 
    \end{aligned}
\end{equation}
whereas 
\begin{equation}
    \begin{aligned}
        \frac{\partial}{\partial y} (v u_y) &= \frac{\left(v(x,y)+v(x,y+h_y)\right)}{2 h_y} \frac{\left(u(x,y+h_y)-u(x,y)\right)}{h_y} \\
        &- \frac{\left(v(x,y)+v(x,y-h_y)\right)}{2 h_y} \frac{\left(u(x,y)-u(x,y-h_y)\right)}{h_y} 
    \end{aligned}
\end{equation}
holds along the $y$ direction. Finally the chemotactic term $- \, div(u \nabla v)$ is discretized as
$$-\frac{\partial^2 (u\,v)}{\partial x^2} -\frac{\partial^2 (u\,v)}{\partial y^2} + \frac{\partial}{\partial x} (v u_x) + \frac{\partial}{\partial y} (v u_y).$$


\bibliographystyle{unsrt} 
\bibliography{biblio}

\begin{thebibliography}{10}

\bibitem{Ramesh2019}
Thangavel Ramesh, Nanthi~S. Bolan, Mary~Beth Kirkham, Hasintha Wijesekara,
  Manjaiah Kanchikerimath, Cherukumalli {Srinivasa Rao}, Sasidharan Sandeep,
  Jörg Rinklebe, Yong~Sik Ok, Burhan~U. Choudhury, Hailong Wang, Caixian Tang,
  Xiaojuan Wang, Zhaoliang Song, and Oliver W.{Freeman II}.
\newblock Chapter one - soil organic carbon dynamics: Impact of land use
  changes and management practices: A review.
\newblock volume 156 of {\em Advances in Agronomy}, pages 1--107. Academic
  Press, 2019.

\bibitem{Chalchissa2024}
Fedhasa~Benti Chalchissa and Birhanu~Kebede Kuris.
\newblock Modelling soil organic carbon dynamics under extreme climate and land
  use and land cover changes in {W}estern {O}romia {R}egional state,
  {E}thiopia.
\newblock {\em Journal of Environmental Management}, 350:119598, 2024.

\bibitem{hammoudi2015mathematical}
Alaaeddine Hammoudi, Oana Iosifescu, and Martial Bernoux.
\newblock Mathematical analysis of a nonlinear model of soil carbon dynamics.
\newblock {\em Differential Equations and Dynamical Systems}, 23(4):453--466,
  2015.

\bibitem{Pansu2010}
Marc Pansu, Lina Sarmiento, Maria Rujano, Magdiel Ablan, D.~Acevedo, and
  P.~Bottner.
\newblock Modeling organic transformations by microorganisms of soils in six
  contrasting ecosystems: Validation of the {MOMOS} model.
\newblock {\em Global Biogeochemical Cycles}, 24, 03 2010.

\bibitem{Pansu2004}
Marc Pansu, Pierre Bottner, Lina Sarmiento, and K.~Metselaar.
\newblock Comparison of five soil organic matter decomposition models using
  data from a 14c and 15n labeling field experiment.
\newblock {\em Global Biogeochemical Cycles}, 18, 12 2004.

\bibitem{Pagel2020}
Holger Pagel, Björn Kriesche, Marie Uksa, Christian Poll, Ellen Kandeler,
  Volker Schmidt, and Thilo Streck.
\newblock Spatial control of carbon dynamics in soil by microbial decomposer
  communities.
\newblock {\em Frontiers in Environmental Science}, 8, 01 2020.

\bibitem{Weigh2023}
Katherine Weigh, Bruna Batista, Huong Hoang, and Paul Dennis.
\newblock Characterisation of soil bacterial communities that exhibit
  chemotaxis to root exudates from phosphorus-limited plants.
\newblock {\em Microorganisms}, 11:2984, 12 2023.

\bibitem{Patlak1953}
Clifford~S. Patlak.
\newblock Random walk with persistence and external bias.
\newblock {\em The bulletin of mathematical biophysics}, 15(3):311--338, 1953.

\bibitem{keller1970initiation}
Evelyn~F Keller and Lee~A Segel.
\newblock Initiation of slime mold aggregation viewed as an instability.
\newblock {\em Journal of theoretical biology}, 26(3):399--415, 1970.

\bibitem{Hillen2008}
Thomas Hillen and Kevin Painter.
\newblock A user’s guide to pde models for chemotaxis.
\newblock {\em Journal of Mathematical Biology}, 58:183--217, 08 2008.

\bibitem{Horstman2003}
D.~Horstman.
\newblock From 1970 until now: the keller–segal model in chemotaxis and its
  consequence i.
\newblock {\em Jahresber. Dtsch. Math.}, 105:103--165, 2003.

\bibitem{Murray_book2}
James~Dickson Murray.
\newblock {\em Mathematical biology: II: spatial models and biomedical
  applications}, volume~18.
\newblock Springer, 2003.

\bibitem{Maini_1997}
P.K. Maini, K.J. Painter, and H.N.P. Chau.
\newblock Spatial pattern formation in chemical and biological systems.
\newblock {\em Journal of the Chemical Society, Faraday Transactions},
  93:3601--3610, 1997.

\bibitem{Zhang_2020}
Jia-Fang Zhang, Hong-Bo Shi, and Anotida Madzvamuse.
\newblock Characterizing the effects of self- and cross-diffusion on stationary
  patterns of a predator–prey system.
\newblock {\em International Journal of Bifurcation and Chaos}, 30(03):2050041,
  2020.

\bibitem{Capone_2021}
F.~Capone, M.F. Carfora, R.~De~Luca, and I.~Torcicollo.
\newblock Nonlinear stability and numerical simulations for a
  reaction–diffusion system modelling allee effect on predators.
\newblock {\em International Journal of Nonlinear Sciences and Numerical
  Simulation}, 23:000010151520200015, 08 2021.

\bibitem{Lacitignola_2022}
D.~Lacitignola, M.~Frittelli, V.~Cusimano, and A.~De~Gaetano.
\newblock Pattern formation on a growing oblate spheroid. an application to
  adult sea urchin development.
\newblock {\em Journal of Computational Dynamics}, 9(2):185--206, 2022.

\bibitem{Grifo_2023}
G.~Grifó, G.~Consolo, C.~Curró, and G.~Valenti.
\newblock Rhombic and hexagonal pattern formation in 2d hyperbolic
  reaction–transport systems in the context of dryland ecology.
\newblock {\em Physica D: Nonlinear Phenomena}, 449:133745, 2023.

\bibitem{Mehdaoui_2024}
M.~Mehdaoui, D.~Lacitignola, and M.~Tilioua.
\newblock Optimal social distancing through cross-diffusion control for a
  disease outbreak pde model.
\newblock {\em Communications in Nonlinear Science and Numerical Simulation},
  131:107855, 2024.

\bibitem{Lacitignola_2021}
D.~Lacitignola, I.~Sgura, and B.~Bozzini.
\newblock Turing-hopf patterns in a morphochemical model for electrodeposition
  with cross-diffusion.
\newblock {\em Applications in Engineering Science}, 5:100034, 2021.

\bibitem{Chaplain_1995}
M.~Chaplain, T.~Höfer, P.~Maini, J.~Sherratt, and J.~Murray.
\newblock Resolving the chemotactic wave paradox: a mathematical model for
  chemotaxis of dictyostelium amoebae.
\newblock {\em J Biol Syst}, 3, 12 1995.

\bibitem{Giunta_2021}
V.~Giunta, M.C. Lombardo, and M.~Sammartino.
\newblock Pattern formation and transition to chaos in a chemotaxis model of
  acute inflammation.
\newblock {\em SIAM Journal on Applied Dynamical Systems}, 20(4):1844--1881,
  2021.

\bibitem{Bellomo_2022}
N.~Bellomo, N.~Outada, J.~Soler, Y.~Tao, and M.~Winkler.
\newblock Chemotaxis and cross-diffusion models in complex environments: Models
  and analytic problems toward a multiscale vision.
\newblock {\em Mathematical Models and Methods in Applied Sciences},
  32(04):713--792, 2022.

\bibitem{Bisi_2023}
M.~Bisi, M.~Groppi, G.~Martal{\`o}, and C.~Soresina.
\newblock A chemotaxis reaction–diffusion model for multiple sclerosis with
  allee effect.
\newblock {\em Ricerche di Matematica}, pages 1--18, 2023.

\bibitem{hammoudi2018mathematical}
Alaaeddine Hammoudi and Oana Iosifescu.
\newblock Mathematical analysis of a chemotaxis-type model of soil carbon
  dynamic.
\newblock {\em Chinese Annals of Mathematics, Series B}, 39(2):253--280, 2018.

\bibitem{MIMURA1996}
Masayasu Mimura and Tohru Tsujikawa.
\newblock Aggregating pattern dynamics in a chemotaxis model including growth.
\newblock {\em Physica A: Statistical Mechanics and its Applications},
  230(3):499--543, 1996.

\bibitem{kuto2012}
Kousuke Kuto, Koichi Osaki, Tatsunari Sakurai, and Tohru Tsujikawa.
\newblock Spatial pattern formation in a chemotaxis–diffusion–growth model.
\newblock {\em Physica D: Nonlinear Phenomena}, 241(19):1629--1639, 2012.

\bibitem{Sarker_2022}
T.C. Sarker, M.~Zotti, Y.~Fang, F.~Giannino, S.~Mazzoleni, G.~Bonanomi, Y.~Cai,
  and S.~Chang.
\newblock Soil aggregation in relation to organic amendment: a synthesis.
\newblock {\em Journal of Soil Science and Plant Nutrition}, 22, 03 2022.

\bibitem{Vogel2014}
Cordula Vogel, Carsten Mueller, Carmen Hoeschen, Franz Buegger, Katja Heister,
  Stefanie Schulz, Michael Schloter, and Ingrid Kögel-Knabner.
\newblock Submicron structures provide preferential spots for carbon and
  nitrogen sequestration in soils.
\newblock {\em Nature communications}, 5:2947, 01 2014.

\bibitem{Schmid_2010}
Peter~J. Schmid.
\newblock Dynamic mode decomposition of numerical and experimental data.
\newblock {\em Journal of Fluid Mechanics}, 656, 2010.

\bibitem{Tu_2014}
Jonathan~H. Tu, Clarence~W. Rowley, Dirk~M. Luchtenburg, Steven~L. Brunton, and
  J.~Nathan Kutz.
\newblock On dynamic mode decomposition: Theory and applications.
\newblock {\em Journal of Computational Dynamics}, 1(2):391--421, 2014.

\bibitem{AMS2024}
Alessandro Alla, Angela Monti, and Ivonne Sgura.
\newblock Piecewise {DMD} for oscillatory and {T}uring spatio-temporal
  dynamics.
\newblock {\em Computers \& Mathematics with Applications}, 160:108--124, 2024.

\bibitem{diele2023soc}
Fasma Diele, Ilenia Luiso, Carmela Marangi, and Angela Martiradonna.
\newblock {SOC}-reactivity analysis for a newly defined class of
  two-dimensional soil organic carbon dynamics.
\newblock {\em Applied Mathematical Modelling}, 2023.

\bibitem{walker2023visualpde}
Benjamin~J Walker, Adam~K Townsend, Alexander~K Chudasama, and Andrew~L Krause.
\newblock Visualpde: rapid interactive simulations of partial differential
  equations.
\newblock {\em Bulletin of Mathematical Biology}, 85(11):113, 2023.

\bibitem{beck2015positivity}
M{\'e}lanie Beck and Martin~J Gander.
\newblock On the positivity of poisson integrators for the lotka--volterra
  equations.
\newblock {\em BIT Numerical Mathematics}, 55:319--340, 2015.

\bibitem{diele2018positive}
Fasma Diele and Carmela Marangi.
\newblock Positive symplectic integrators for predator-prey dynamics.
\newblock {\em Discrete \& Continuous Dynamical Systems-B}, 23(7):2661, 2018.

\bibitem{hairer2006geometric}
Ernst Hairer, Marlis Hochbruck, Arieh Iserles, and Christian Lubich.
\newblock Geometric numerical integration.
\newblock {\em Oberwolfach Reports}, 3(1):805--882, 2006.

\bibitem{diele2019geometric}
Fasma Diele and Carmela Marangi.
\newblock Geometric numerical integration in ecological modelling.
\newblock {\em Mathematics}, 8(1):25, 2019.

\bibitem{settanni2016devising}
Giuseppina Settanni and Ivonne Sgura.
\newblock Devising efficient numerical methods for oscillating patterns in
  reaction--diffusion systems.
\newblock {\em Journal of Computational and Applied Mathematics}, 292:674--693,
  2016.

\bibitem{diele2014imsp}
Fasma Diele, Carmela Marangi, and Stefania Ragni.
\newblock Imsp schemes for spatially explicit models of cyclic populations and
  metapopulation dynamics.
\newblock {\em Mathematics and Computers in Simulation}, 100:41--53, 2014.

\bibitem{diele2017numerical}
Fasma Diele, Marcus Garvie, and Catalin Trenchea.
\newblock Numerical analysis of a first-order in time implicit-symplectic
  scheme for predator--prey systems.
\newblock {\em Computers \& Mathematics with Applications}, 74(5):948--961,
  2017.

\bibitem{diele2022}
Fasma Diele, Angela Martiradonna, and Catalin Trenchea.
\newblock Stability and errors estimates of a second-order imsp scheme.
\newblock {\em Discrete and Continuous Dynamical Systems - S},
  15(12):3645--3665, 2022.

\bibitem{HH17}
P.~Héas and C.~Herzet.
\newblock Optimal low-rank dynamic mode decomposition.
\newblock In {\em 2017 IEEE International Conference on Acoustics, Speech and
  Signal Processing (ICASSP)}, pages 4456--4460, 2017.

\bibitem{AK2018}
T.~Askham and J.~N. Kutz.
\newblock Variable projection methods for an optimized dynamic mode
  decomposition.
\newblock {\em SIAM Journal on Applied Dynamical Systems}, 17(1):380--416,
  2018.

\bibitem{EDMD2015}
M.~O. Williams, I.~G. Kevrekidis, and C.~W. Rowley.
\newblock A data--driven approximation of the koopman operator: Extending
  dynamic mode decomposition.
\newblock {\em Journal of Nonlinear Science}, 25:1307--1346, 2015.

\bibitem{LV17}
S.~Le~Clainche and J.~M. Vega.
\newblock Higher order dynamic mode decomposition.
\newblock {\em SIAM Journal on Applied Dynamical Systems}, 16(2):882--925,
  2017.

\bibitem{VBGRC22}
A.~Viguerie, G.~F. Barros, M.~Grave, A.~Reali, and A.~L. G.~A. Coutinho.
\newblock Coupled and uncoupled dynamic mode decomposition in
  multi-compartmental systems with applications to epidemiological and additive
  manufacturing problems.
\newblock {\em Computer Methods in Applied Mechanics and Engineering},
  391:114600, 2022.

\bibitem{EMKB2019}
N.~B. Erichson, L.~Mathelin, J.~N. Kutz, and S.~L. Brunton.
\newblock Randomized {D}ynamic {M}ode {D}ecomposition.
\newblock {\em SIAM Journal on Applied Dynamical Systems}, 18(4):1867--1891,
  2019.

\bibitem{sgura2012numerical}
Ivonne Sgura, Benedetto Bozzini, and Deborah Lacitignola.
\newblock Numerical approximation of turing patterns in electrodeposition by
  adi methods.
\newblock {\em Journal of Computational and Applied Mathematics},
  236(16):4132--4147, 2012.

\bibitem{Patoine_2022}
Guillaume Patoine, Nico Eisenhauer, Simone Cesarz, Helen Phillips, Xiaofeng Xu,
  Lihua Zhang, and Carlos Guerra.
\newblock Drivers and trends of global soil microbial carbon over two decades.
\newblock {\em Nature Communications}, 13, 07 2022.

\bibitem{Bao}
T.~Bao, Y.~Qian, Y~Xin, J.~Collins, and T.~Lu.
\newblock Engineering microbial division of labor for plastic upcycling.
\newblock {\em Nature Communications}, 14:5712, 2013.

\bibitem{Singh_2024}
Shalini Singh, Boda Kiran, and S~Venkata~Mohan.
\newblock Carbon farming: a circular framework to augment co2 sinks and to
  combat climate change.
\newblock {\em Environmental Science: Advances}, 3, 02 2024.

\end{thebibliography}





\end{document}